\documentclass[a4paper]{amsart}%
\usepackage{amsfonts}
\usepackage{amssymb,latexsym}
\usepackage{amsmath}
\usepackage{amssymb}
\usepackage{graphicx}%
\theoremstyle{plain}
\newtheorem{theorem}{Theorem}
\newtheorem{corollary}[theorem]{Corollary}
\newtheorem{lemma}[theorem]{Lemma}
\theoremstyle{remark}
\newtheorem*{remark}{Remark}
\newtheorem*{example}{Example}
\numberwithin{equation}{section}
\numberwithin{theorem}{section}
\begin{document}
\title[Minimal clones with few majority operations]{Minimal clones with few majority operations}
\author[T. Waldhauser]{Tam\'{a}s Waldhauser}
\address{Bolyai Institute\\
University of Szeged\\
Aradi v\'{e}rtan\'{u}k tere 1, H--6720, Szeged, Hungary}
\email{twaldha@math.u-szeged.hu}
\urladdr{http://www.math.u-szeged.hu/\symbol{126}twaldha}
\thanks{{Research supported by the Hungarian National Foundation for Scientific
Research grant no. T48809 and K60148.}}
\keywords{{clone, minimal clone, majority operation}}
\subjclass[2000]{08A40}

\begin{abstract}
We characterize minimal clones generated by a majority function containing at
most seven ternary operations.

\end{abstract}
\dedicatory{Dedicated to B\'{e}la Cs\'{a}k\'{a}ny on his seventy-fifth birthday}
\maketitle

\section{Introduction\label{sect intro}}

A set $\mathcal{C}$ of finitary operations on a set $A$ is a \emph{(concrete)
clone}, if it is closed under composition of functions and contains all
projections. If $\mathbb{A}=\left(  A;F\right)  $ is an algebra, then the set
of its term functions, denoted by $\operatorname{Clo}\mathbb{A}$, is a clone
on $A$, called the \emph{clone of the algebra} $\mathbb{A}$. This is the
smallest clone containing $F$, therefore we say that $F$ \emph{generates}
$\operatorname{Clo}\mathbb{A}$, and we write $\left[  F\right]
=\operatorname{Clo}\mathbb{A}$. Clearly, every clone arises as the clone of an
algebra: we just need to pick a generating set for the clone, and let these be
the basic operations of the algebra.

An \emph{(abstract) clone} is a heterogeneous algebra that captures the
compositional structure of concrete clones \cite{BL,Taylor}. More precisely,
an abstract clone $\mathcal{C}$ is given by a family $\mathcal{C}^{\left(
n\right)  }\left(  n\geq1\right)  $ of sets with distinguished elements
$e_{i}^{\left(  n\right)  }\in\mathcal{C}^{\left(  n\right)  }\left(  1\leq
i\leq n\right)  $ and mappings
\[
F_{k}^{n}:\mathcal{C}^{\left(  n\right)  }\times\left(  \mathcal{C}^{\left(
k\right)  }\right)  ^{n}\rightarrow\mathcal{C}^{\left(  k\right)  },\,\left(
f,g_{1},\ldots,g_{n}\right)  \mapsto f\left(  g_{1},\ldots,g_{n}\right)
\quad\left(  n,k\geq1\right)  ,
\]
such that the following three axioms are satisfied for all $f\in
\mathcal{C}^{\left(  n\right)  },\,g_{1},\ldots,g_{n}\in\mathcal{C}^{\left(
k\right)  },$\linebreak$h_{1},\ldots,h_{k}\in\mathcal{C}^{\left(  l\right)
}\left(  n,k,l\geq1\right)  $:%
\begin{align*}
e_{i}^{\left(  n\right)  }\left(  g_{1},\ldots,g_{n}\right)   &  =g_{i}%
\quad\left(  i=1,\ldots,n\right)  ,\\
f\left(  e_{1}^{\left(  n\right)  },\ldots,e_{n}^{\left(  n\right)  }\right)
&  =f,\\
f\left(  g_{1},\ldots,g_{n}\right)  \left(  h_{1},\ldots,h_{k}\right)   &
=f\left(  g_{1}\left(  h_{1},\ldots,h_{k}\right)  ,\ldots,g_{n}\left(
h_{1},\ldots,h_{k}\right)  \right)  .
\end{align*}
The notion of a subclone, clone homomorphism and factor clone can be defined
in a natural way, and the isomorphism theorems can be proved for abstract clones.

Every concrete clone can be regarded as an abstract clone if we let
$e_{i}^{\left(  n\right)  }$ be the $i$-th $n$-ary projection, and $F_{k}%
^{n}\left(  f,g_{1},\ldots,g_{n}\right)  $ be the composition of $f$ by
$g_{1},\ldots,g_{n}$, as we have already indicated it in the notation. We will
call the elements $e_{i}^{\left(  n\right)  }$ \emph{projections}, the
mappings $F_{k}^{n}$ \emph{composition} operations, and $\mathcal{C}^{\left(
n\right)  }$ the $n$-ary part of $\mathcal{C}$, even if the elements of the
abstract clone are not functions. However, every abstract clone is isomorphic
to a concrete clone, so we can always assume that the elements of the clone
are actually functions.

There is a close relationship between abstract clones and varieties; roughly
speaking, abstract clones are the same as varieties up to term-equivalence
\cite{KK,LLPPP}. To explain this more explicitly, let us fix an abstract clone
$\mathcal{C}$, and a generating set $F$ of $\mathcal{C}$. For any clone
homomorphism $\varphi$ from $\mathcal{C}$ to a concrete clone on some set $A$
we can construct the algebra $\mathbb{A}=\left(  A;\varphi\left(  F\right)
\right)  $ whose clone is $\varphi\left(  \mathcal{C}\right)  $. It is not
hard to see that the algebras arising this way form a variety. If we choose
another set of generators, then we get another variety which is
term-equivalent to the previous one. Conversely, a clone can be assigned to
every variety, namely the clone of the countably generated free algebra of the
variety, and these two correspondences between clones and varieties are
inverses of each other (up to isomorphism of clones and term-equivalence of varieties).

If $\mathcal{C}$ and $\mathcal{V}$ correspond to each other, then subvarieties
of $\mathcal{V}$ correspond to factor clones of $\mathcal{C}$, and the
congruence lattice of $\mathcal{C}$ is dually isomorphic to the subvariety
lattice of $\mathcal{V}$. If $\mathcal{V}$ is generated by $\mathbb{A}$, then
$\mathcal{C}\cong\operatorname{Clo}\mathbb{A}$, and an algebra $\mathbb{B}$
(of the appropriate type) belongs to $\mathcal{V}$ iff $\operatorname{Clo}%
\mathbb{B}$ is a homomorphic image of $\operatorname{Clo}\mathbb{A}$. An
important special case is when $\mathbb{B}$ is a subalgebra of $\mathbb{A}$.
In this case the restriction map $\operatorname{Clo}\mathbb{A}\rightarrow
\operatorname{Clo}\mathbb{B},\,f\mapsto f|_{B}$ is a surjective clone homomorphism.

All clones on a given set $A$ form a lattice with respect to inclusion; the
smallest element of this lattice is the \emph{trivial clone}, the clone of all
projections on $A$, while the greatest element is the clone of all finitary
operations on $A$. These clones will be denoted by $\mathcal{I}_{A}$ and
$\mathcal{O}_{A}$ respectively. The elements of the trivial clone (the
projections) will be referred to as \emph{trivial functions}. An abstract
clone $\mathcal{C}$ is called trivial if $\mathcal{C}^{\left(  n\right)
}=\left\{  e_{1}^{\left(  n\right)  },\ldots,e_{n}^{\left(  n\right)
}\right\}  $ for all $n\geq1$.

We say that $\mathcal{C}$ is a \emph{minimal clone}, if it has exactly two
subclones: the trivial clone and $\mathcal{C}$ itself. In the case of concrete
clones on a given set $A$, we can identify minimal clones as the atoms of the
clone lattice. On finite sets there are finitely many minimal clones, and
every clone contains a minimal one (cf. \cite{PK,Qsurv,SzAclUA}).

Clearly, a nontrivial clone is minimal iff it is generated by any of its
nontrivial elements. Therefore all minimal clones are one-generated, thus they
arise as clones of algebras with a single basic operation. It is convenient to
choose a function of the least possible arity as a generator of a minimal
clone. These generators are called \emph{minimal functions}: $f$ is a minimal
function iff $\left[  f\right]  $ is a minimal clone and there is no
nontrivial function in $\left[  f\right]  $ whose arity is less than the arity
of $f$. According to \'{S}wierczkowski's lemma \cite{Sw}, a minimal function
must be either a unary operation, a binary idempotent operation, a ternary
majority or minority operation or a semiprojection. Rosenberg's theorem
\cite{R5typ,SzAclUA} characterizes minimal unary and ternary minority
operations, but for the other three types a general description of minimal
functions (or clones) seems to be far beyond reach.

There are numerous partial results that describe minimal clones or minimal
functions under certain assumptions, and the goal of this paper is to prove a
new theorem of this kind. In the next section we recall only a few facts about
minimal clones that we will need in the sequel; for a survey of minimal clones
see \cite{Csminicourse} and \cite{Qsurv}. In Section~\ref{sect symm maj} we
prove a theorem about the possible symmetries of majority functions in a
minimal clone (Theorem~\ref{thm symm}), and in Section~\ref{sect atmostfour}
we use this theorem to obtain a characterization of those minimal clones
generated by a majority operation which contain at most seven ternary
operations (Theorem~\ref{thm few}).

\section{Preliminaries\label{sect prel}}

For brevity we will say that $\mathcal{C}$ is a \emph{majority clone} if
$\mathcal{C=}\left[  f\right]  $ where $f$ is a \emph{majority operation},
i.e. $f$ satisfies the identities%
\[
f\left(  x,x,y\right)  =f\left(  x,y,x\right)  =f\left(  y,x,x\right)  =x.
\]

First we state and prove a very special property of majority clones. This fact
seems to be folklore; usually it is derived from Rosenberg's theorem or from
\'{S}wierczkowski's lemma. Here we give an almost self-contained proof.

\begin{theorem}
[\cite{Cscons}]\label{thm maj min iff 3min}Let $\mathcal{C}$ be a clone
generated by a majority operation $f$. If every majority operation in
$\mathcal{C}$ generates $f$, then $\mathcal{C}$ is a minimal clone.
\end{theorem}

\begin{proof}
The key is the following observation, which can be proved by a simple
induction argument \cite{Cscons}. If $g$ is a nontrivial operation in a clone
generated by a majority function, then $g$ is a so-called \emph{near-unanimity
operation}, i.e. it satisfies the identities
\[
g(y,x,x,\ldots,x,x)=g(x,y,x,\ldots,x,x)=\cdots=g(x,x,x,\ldots,x,y)=x.
\]

We show that any near-unanimity function $g$ of arity $n\geq4$ produces a
nontrivial function of arity $n-1$ by a suitable identification of its
variables. Let us suppose that $g\left(  x,x,x_{3},\ldots,x_{n}\right)  $ is a
projection. Identifying all the $x_{i}$s except for $x_{n}$ with $x$, we get a
projection onto $x$ by the near-unanimity property, therefore $g\left(
x,x,x_{3},\ldots,x_{n}\right)  $ cannot be a projection onto $x_{n}$. This can
be done for any $x_{i}$ instead of $x_{n}$, thus $g\left(  x,x,x_{3}%
,\ldots,x_{n}\right)  $ has to be a projection onto $x$. If we suppose that
$g\left(  x_{1},x_{2},y,y,x_{5},\ldots,x_{n}\right)  $ is also a projection,
then a similar argument shows that it must be a projection onto $y$. Now we
have a contradiction, because $g\left(  x,x,y,y,x_{5},\ldots,x_{n}\right)  $
is a projection to $x$ and $y$ at the same time (this is where we use that
$n\geq4$). Thus we have proved that either $g\left(  x,x,x_{3},\ldots
,x_{n}\right)  $ or $g\left(  x_{1},x_{2},y,y,x_{5},\ldots,x_{n}\right)  $ is nontrivial.

Now if $g$ is an at least quaternary near-unanimity function in the clone
$\mathcal{C}$, then it produces a nontrivial function of arity one less, which
is again a near-unanimity function, since it is still generated by $f$. Hence
if this new function is still of arity at least $4$, then it produces a
near-unanimity function of lesser arity too, and we can continue this way
until we end up with a near-unanimity function of arity $3$, i.e. a majority
operation. Since it was supposed that every majority operation in
$\mathcal{C}$ generates $f$, we have $f\in\left[  g\right]  $, hence
$\mathcal{C}=\left[  g\right]  $, and this shows that $\mathcal{C}$ is a
minimal clone.
\end{proof}

The advantage of this property is that in order to prove the minimality of a
majority clone it suffices to consider the ternary part of the clone. On a
finite set this means a finite number of functions, while in the binary and
semiprojection case one has to consider infinitely many functions.

Restricting our attention to the ternary operations of an abstract clone we
get the algebra $\left(  \mathcal{C}^{\left(  3\right)  };F_{3}^{3}%
,e_{1}^{\left(  3\right)  },e_{2}^{\left(  3\right)  },e_{3}^{\left(
3\right)  }\right)  $. We will refer to this algebra as the ternary part of
$\mathcal{C}$, and denote it briefly by $\mathcal{C}^{\left(  3\right)  }$.
This is an algebra with one quaternary and three nullary operations satisfying
the following identities.%
\begin{align*}
F_{3}^{3}\left(  e_{i}^{\left(  3\right)  },f_{1},f_{2},f_{3}\right)   &
=f_{i}\quad\left(  i=1,2,3\right) \\
F_{3}^{3}\left(  f,e_{1}^{\left(  3\right)  },e_{2}^{\left(  3\right)  }%
,e_{3}^{\left(  3\right)  }\right)   &  =f\\
F_{3}^{3}\left(  F_{3}^{3}\left(  f,g_{1},g_{2},g_{3}\right)  ,h_{1}%
,h_{2},h_{3}\right)   &  =\\
F_{3}^{3}(f,F_{3}^{3}  &  \left(  g_{1},h_{1},h_{2},h_{3}\right)  ,F_{3}%
^{3}\left(  g_{2},h_{1},h_{2},h_{3}\right)  ,F_{3}^{3}\left(  g_{3}%
,h_{1},h_{2},h_{3}\right)  )
\end{align*}
Now Theorem~\ref{thm maj min iff 3min} can be formulated in the following way:
A majority clone $\mathcal{C}$ is minimal iff $\left\{  e_{1}^{\left(
3\right)  },e_{2}^{\left(  3\right)  },e_{3}^{\left(  3\right)  }\right\}  $
is the only proper subalgebra of $\mathcal{C}^{\left(  3\right)  }$.

As opposed to the case of binary operations and semiprojections, there are not
many examples of minimal majority functions. The simplest ones are the median
function $\left(  x\wedge y\right)  \vee\left(  y\wedge z\right)  \vee\left(
z\wedge x\right)  $ on any lattice \cite{PK}, and the dual discriminator
function on any set \cite{CsG,FP}. The description of minimal clones on the
three-element set given by B.~Cs\'{a}k\'{a}ny \cite{Cs3all} yields some more examples.

\begin{theorem}
[\cite{Cs3all}]\label{thm Cs3 maj}If $f$ is a minimal majority function on a
three-element set, then $f$ is isomorphic to one of the twelve majority
functions shown in Table~\textup{\ref{table 3maj}}. These functions belong to
three minimal clones containing $1,3$ and $8$ majority operations
respectively, as shown in the table.
\end{theorem}

\begin{table}[h]
\centering%
\begin{tabular}
[c]{ccccccccccccc}\hline
\multicolumn{1}{|c}{} & \multicolumn{1}{|c}{$m_{1}$} &
\multicolumn{1}{|c}{$m_{2}$} &  &  & \multicolumn{1}{|c}{$m_{3}$} &  &  &  &
&  &  & \multicolumn{1}{c|}{}\\\hline
\multicolumn{1}{|c}{$(1,2,3)$} & \multicolumn{1}{|c}{$1$} &
\multicolumn{1}{|c}{$1$} & $2$ & $3$ & \multicolumn{1}{|c}{$3$} & $3$ & $1$ &
$3$ & $1$ & $1$ & $3$ & \multicolumn{1}{c|}{$1$}\\
\multicolumn{1}{|c}{$(2,3,1)$} & \multicolumn{1}{|c}{$1$} &
\multicolumn{1}{|c}{$2$} & $3$ & $1$ & \multicolumn{1}{|c}{$3$} & $1$ & $3$ &
$3$ & $1$ & $3$ & $1$ & \multicolumn{1}{c|}{$1$}\\
\multicolumn{1}{|c}{$(3,1,2)$} & \multicolumn{1}{|c}{$1$} &
\multicolumn{1}{|c}{$3$} & $1$ & $2$ & \multicolumn{1}{|c}{$3$} & $3$ & $3$ &
$1$ & $1$ & $1$ & $1$ & \multicolumn{1}{c|}{$3$}\\\hline
\multicolumn{1}{|c}{$(2,1,3)$} & \multicolumn{1}{|c}{$1$} &
\multicolumn{1}{|c}{$2$} & $1$ & $3$ & \multicolumn{1}{|c}{$1$} & $3$ & $1$ &
$1$ & $3$ & $1$ & $3$ & \multicolumn{1}{c|}{$3$}\\
\multicolumn{1}{|c}{$(1,3,2)$} & \multicolumn{1}{|c}{$1$} &
\multicolumn{1}{|c}{$1$} & $3$ & $2$ & \multicolumn{1}{|c}{$1$} & $1$ & $1$ &
$3$ & $3$ & $3$ & $3$ & \multicolumn{1}{c|}{$1$}\\
\multicolumn{1}{|c}{$(3,2,1)$} & \multicolumn{1}{|c}{$1$} &
\multicolumn{1}{|c}{$3$} & $2$ & $1$ & \multicolumn{1}{|c}{$1$} & $1$ & $3$ &
$1$ & $3$ & $3$ & $1$ & \multicolumn{1}{c|}{$3$}\\\hline
&  &  &  &  &  &  &  &  &  &  &  &
\end{tabular}
\caption{Minimal majority functions on the three-element set}%
\label{table 3maj}%
\end{table}

Note that we have omitted those triples in the table where the majority rule
determines the value of the functions. Let us also observe that $m_{1}$ can be
defined as the median function of the three-element chain (with the unusual
order $2<1<3$ or $3<1<2$), and $m_{2}$ is nothing else but the dual
discriminator, up to a permutation of variables (the third function in
$\left[  m_{2}\right]  $ is actually the dual discriminator).

Based on this theorem, B.~Cs\'{a}k\'{a}ny obtained a characterization of
minimal majority operations which are \emph{conservative} \cite{Cscons}. A
function is \emph{conservative} if it preserves every subset of the underlying
set (cf. \cite{Qcat}). It is clear that if $f$ is a conservative minimal
majority function on a set $A$, and $B\subseteq A$ is a three-element subset,
then $f|_{B}$ is a minimal majority function on $B$. Thus $f|_{B}$ is
isomorphic to one of the above twelve functions. These restrictions determine
$f$, so we can say that $f$ is somehow glued together from copies of the
functions listed in Table~\ref{table 3maj}.

We do not quote the result here, but let us note that from this description it
follows that there are only four conservative minimal majority clones up to
isomorphism of the ternary part of the clone (but not up to isomorphism of the
whole clone; see the example in Section~\ref{sect symm maj}). For three of
them the ternary part is isomorphic to $\left[  m_{1}\right]  ^{\left(
3\right)  },\left[  m_{2}\right]  ^{\left(  3\right)  },\left[  m_{3}\right]
^{\left(  3\right)  }$ respectively, hence they contain $1,3$ and $8$ majority
operations, while the fourth one contains $24$ majority operations.

As the next theorem shows, the nonconservative minimal majority functions on
the four-element set are quite similar to those on the three-element set.

\begin{theorem}
[\cite{W4maj}]\label{thm 3.14}If $f$ is a minimal majority function on a
four-element set, then $f$ is either conservative, or isomorphic to one of the
twelve majority functions shown in Table~\textup{\ref{table 4maj}}. These
functions belong to three minimal clones containing $1,3$ and $8$ majority
operations respectively, as shown in the table. \textup{(}The middle two rows
mean that if $\left\{  a,b,c\right\}  $ equals $\left\{  1,2,4\right\}  $ or
$\left\{  1,3,4\right\}  $, then the value of the functions on $\left(
a,b,c\right)  $ is $4$.\textup{) }Moreover, the clone generated by $M_{i}$ is
isomorphic to $\left[  m_{i}\right]  $.
\end{theorem}

\begin{table}[h]
\centering%
\begin{tabular}
[c]{ccccccccccccc}\hline
\multicolumn{1}{|c}{} & \multicolumn{1}{|c}{$M_{1}$} &
\multicolumn{1}{|c}{$M_{2}$} &  &  & \multicolumn{1}{|c}{$M_{3}$} &  &  &  &
&  &  & \multicolumn{1}{c|}{}\\\hline
\multicolumn{1}{|c}{$(1,2,3)$} & \multicolumn{1}{|c}{$4$} &
\multicolumn{1}{|c}{$4$} & $2$ & $3$ & \multicolumn{1}{|c}{$3$} & $3$ & $4$ &
$3$ & $4$ & $4$ & $3$ & \multicolumn{1}{c|}{$4$}\\
\multicolumn{1}{|c}{$(2,3,1)$} & \multicolumn{1}{|c}{$4$} &
\multicolumn{1}{|c}{$2$} & $3$ & $4$ & \multicolumn{1}{|c}{$3$} & $4$ & $3$ &
$3$ & $4$ & $3$ & $4$ & \multicolumn{1}{c|}{$4$}\\
\multicolumn{1}{|c}{$(3,1,2)$} & \multicolumn{1}{|c}{$4$} &
\multicolumn{1}{|c}{$3$} & $4$ & $2$ & \multicolumn{1}{|c}{$3$} & $3$ & $3$ &
$4$ & $4$ & $4$ & $4$ & \multicolumn{1}{c|}{$3$}\\\hline
\multicolumn{1}{|c}{$(2,1,3)$} & \multicolumn{1}{|c}{$4$} &
\multicolumn{1}{|c}{$2$} & $4$ & $3$ & \multicolumn{1}{|c}{$4$} & $3$ & $4$ &
$4$ & $3$ & $4$ & $3$ & \multicolumn{1}{c|}{$3$}\\
\multicolumn{1}{|c}{$(1,3,2)$} & \multicolumn{1}{|c}{$4$} &
\multicolumn{1}{|c}{$4$} & $3$ & $2$ & \multicolumn{1}{|c}{$4$} & $4$ & $4$ &
$3$ & $3$ & $3$ & $3$ & \multicolumn{1}{c|}{$4$}\\
\multicolumn{1}{|c}{$(3,2,1)$} & \multicolumn{1}{|c}{$4$} &
\multicolumn{1}{|c}{$3$} & $2$ & $4$ & \multicolumn{1}{|c}{$4$} & $4$ & $3$ &
$4$ & $3$ & $3$ & $4$ & \multicolumn{1}{c|}{$3$}\\\hline
\multicolumn{1}{|c}{$\{1,2,4\}$} & \multicolumn{1}{|c}{$4$} &
\multicolumn{1}{|c}{$4$} & $4$ & $4$ & \multicolumn{1}{|c}{$4$} & $4$ & $4$ &
$4$ & $4$ & $4$ & $4$ & \multicolumn{1}{c|}{$4$}\\\hline
\multicolumn{1}{|c}{$\{1,3,4\}$} & \multicolumn{1}{|c}{$4$} &
\multicolumn{1}{|c}{$4$} & $4$ & $4$ & \multicolumn{1}{|c}{$4$} & $4$ & $4$ &
$4$ & $4$ & $4$ & $4$ & \multicolumn{1}{c|}{$4$}\\\hline
\multicolumn{1}{|c}{$(4,2,3)$} & \multicolumn{1}{|c}{$4$} &
\multicolumn{1}{|c}{$4$} & $2$ & $3$ & \multicolumn{1}{|c}{$3$} & $3$ & $4$ &
$3$ & $4$ & $4$ & $3$ & \multicolumn{1}{c|}{$4$}\\
\multicolumn{1}{|c}{$(2,3,4)$} & \multicolumn{1}{|c}{$4$} &
\multicolumn{1}{|c}{$2$} & $3$ & $4$ & \multicolumn{1}{|c}{$3$} & $4$ & $3$ &
$3$ & $4$ & $3$ & $4$ & \multicolumn{1}{c|}{$4$}\\
\multicolumn{1}{|c}{$(3,4,2)$} & \multicolumn{1}{|c}{$4$} &
\multicolumn{1}{|c}{$3$} & $4$ & $2$ & \multicolumn{1}{|c}{$3$} & $3$ & $3$ &
$4$ & $4$ & $4$ & $4$ & \multicolumn{1}{c|}{$3$}\\\hline
\multicolumn{1}{|c}{$(2,4,3)$} & \multicolumn{1}{|c}{$4$} &
\multicolumn{1}{|c}{$2$} & $4$ & $3$ & \multicolumn{1}{|c}{$4$} & $3$ & $4$ &
$4$ & $3$ & $4$ & $3$ & \multicolumn{1}{c|}{$3$}\\
\multicolumn{1}{|c}{$(4,3,2)$} & \multicolumn{1}{|c}{$4$} &
\multicolumn{1}{|c}{$4$} & $3$ & $2$ & \multicolumn{1}{|c}{$4$} & $4$ & $4$ &
$3$ & $3$ & $3$ & $3$ & \multicolumn{1}{c|}{$4$}\\
\multicolumn{1}{|c}{$(3,2,4)$} & \multicolumn{1}{|c}{$4$} &
\multicolumn{1}{|c}{$3$} & $2$ & $4$ & \multicolumn{1}{|c}{$4$} & $4$ & $3$ &
$4$ & $3$ & $3$ & $4$ & \multicolumn{1}{c|}{$3$}\\\hline
&  &  &  &  &  &  &  &  &  &  &  &
\end{tabular}
\caption{Nonconservative minimal majority functions on the four-element set}%
\label{table 4maj}%
\end{table}

In the above examples, and actually for all known examples of minimal majority
clones, the ternary part of the clone is isomorphic to the ternary part of a
conservative clone. Thus we have only four examples for minimal majority
clones up to isomorphism of the ternary part of the clones, so it is natural
to ask if there are other examples at all. We investigate this question by
describing minimal majority clones with few (at most seven) ternary
operations. It turns out that if $\mathcal{C}$ is such a clone, then
$\mathcal{C}^{\left(  3\right)  }$ is isomorphic to $\left[  m_{1}\right]
^{\left(  3\right)  }$ or $\left[  m_{2}\right]  ^{\left(  3\right)  }$. Let
us note that the binary case is studied in \cite{LLPPP}, where binary minimal
clones with at most seven binary operations are determined.

\section{Symmetries of minimal majority functions\label{sect symm maj}}

For any abstract clone $\mathcal{C}$, the symmetric group $S_{n}$ acts
naturally on $\mathcal{C}^{\left(  n\right)  }$: applying a permutation
$\pi\in S_{n}$ to $f\in\mathcal{C}^{\left(  n\right)  }$ we get
\begin{equation}
f\left(  e_{\pi\left(  1\right)  }^{\left(  n\right)  },e_{\pi\left(
2\right)  }^{\left(  n\right)  },\ldots,e_{\pi\left(  n\right)  }^{\left(
n\right)  }\right)  . \label{permutation}%
\end{equation}
In the case of concrete clones this means that we permute the variables of $f
$, and we will adopt this terminology to the abstract case, even though we
cannot speak about variables here. If $f$ is a nontrivial operation, then so
are the operations of the form (\ref{permutation}), hence $S_{n}$ acts on
$\mathcal{C}^{\left(  3\right)  }\setminus\mathcal{I}$, too. Let us denote by
$\sigma\!\left(  f\right)  $ the stabilizer of $f$, i.e. the group of
permutations leaving $f$ invariant.

If $f$ is a majority operation, then $\sigma\!\left(  f\right)  $ is a
subgroup of $S_{3}$, therefore it has $1,2,3$ or $6$ elements. If
$\sigma\!\left(  f\right)  \supseteq A_{3}$, then we say that $f$ is
\emph{cyclically symmetric}, and if $\sigma\!\left(  f\right)  =S_{3}$, then
we say that $f$ is \emph{totally symmetric.}

If $\mathcal{C}$ is a majority clone with just one majority operation, then
the majority rule and the clone axioms completely determine the structure of
$\mathcal{C}^{\left(  3\right)  }$, and it is clear from
Theorem~\ref{thm maj min iff 3min} that in this case $\mathcal{C}$ is minimal.
For example, $\left[  m_{1}\right]  $ is such a clone, so we have the
following theorem.

\begin{theorem}
\label{thm maj=1}If $\mathcal{C}$ is a minimal clone with one majority
operation, then $\mathcal{C}^{\left(  3\right)  }$ is isomorphic to $\left[
m_{1}\right]  ^{\left(  3\right)  }$.
\end{theorem}

If $f$ is the unique majority operation in such a clone, then every nontrivial
ternary superposition of $f$ yields $f$ itself. In particular, $f$ is totally
symmetric, and satisfies $f\left(  f\left(  x,y,z\right)  ,y,z\right)
=f\left(  x,y,z\right)  $. It is easy to check that this identity together
with the total symmetry ensures that $f$ does not generate any nontrivial
ternary operation other than $f$. Thus the clones described in the above
theorem are exactly the factor clones of the clone of the variety
$\mathcal{M}_{1}$ defined by the following identities:%
\begin{equation}
f\left(  x,y,z\right)  =f\left(  y,z,x\right)  =f\left(  y,x,z\right)
=f\left(  f\left(  x,y,z\right)  ,y,z\right)  ,~f\left(  x,x,y\right)  =x.
\label{e M1}%
\end{equation}

This variety has infinitely many subvarieties, therefore there are infinitely
many nonisomorphic minimal clones with just one majority operation. To show
this, we will construct a subdirectly irreducible (in fact, simple) algebra
$\mathbb{A}_{n}\in\mathcal{M}_{1}$ of size $n$ for every $n>6$. Since
$\mathcal{M}_{1}$ is congruence distributive, $\mathbb{A}_{m}\notin$
$\operatorname{HSP}(\mathbb{A}_{n})$ if $m>n$ by J\'{o}nsson's lemma, hence
the subvarieties $\operatorname{HSP}(\mathbb{A}_{n})$ are all different, and
the clones $\operatorname{Clo}\mathbb{A}_{n}$ are pairwise nonisomorphic.

\begin{example}
Let $\mathbb{A}_{n}=\left(  \left\{  1,2,\ldots,n\right\}  ;f\right)  $, where
$f $ is a totally symmetric majority operation defined for $\,1\leq a<b<c\leq
n$ by%
\[
f\left(  a,b,c\right)  =%
\begin{cases}
a & \text{if }\left\lceil \frac{a+c}{2}\right\rceil <b;\\
b & \text{if }b=\left\lfloor \frac{a+c}{2}\right\rfloor \text{ or
}b=\left\lceil \frac{a+c}{2}\right\rceil ;\\
c & \text{if }b<\left\lfloor \frac{a+c}{2}\right\rfloor .
\end{cases}
\]
Note that it suffices to define $f\left(  a,b,c\right)  $ for $a<b<c$ since
$f$ is a totally symmetric majority function. Let us consider the elements of
$\mathbb{A}_{n}$ as points on the real line. We will call the points
$\left\lfloor \frac{a+c}{2}\right\rfloor $ and $\left\lceil \frac{a+c}%
{2}\right\rceil $ the midpoints of the segment between $a$ and $c$. (Segments
of even length have one midpoint, but segments of odd length have two
midpoints!) If $a<b<c$ and $b$ is a midpoint of the segment between $a$ and
$c$, then $f\left(  a,b,c\right)  =b$, otherwise $f\left(  a,b,c\right)  $ is
that endpoint of this segment which is farther from $b$.

It is easy to check that $\mathbb{A}_{n}\in\mathcal{M}_{1}$ (note that $f$ is
conservative), and we claim that $\mathbb{A}_{n}$ is simple if $n>6$. Let us
first observe that since $f$ is a majority operation, any congruence class $I$
has the following property: if at least two of $a,b,c$ belong to $I $, then
$f\left(  a,b,c\right)  \in I$. Let us call such subsets ideals of
$\mathbb{A}_{n}$. If $I$ is an ideal and $a,c\in I$, then $I$ contains the
midpoints of the segment between $a$ and $c$. Successively taking midpoints we
can reach any point between $a$ and $c$, therefore this whole segment belongs
to $I$, i.e. ideals are convex.

Let $\vartheta$ be a nontrivial congruence of $\mathbb{A}_{n}$, and let $a$ be
the least element of $\mathbb{A}_{n}$ that belongs to a non-singleton block
$I$ of $\vartheta$. Since $a$ is the smallest element of $I$, which is a
convex set with at least two elements, we must have $a+1\in I$. If $a\geq4$,
then $f\left(  1,a,a+1\right)  =1$, and by the ideal property $f\left(
1,a,a+1\right)  \in I$. Now $2\in I$ follows by convexity, and then
$n=f\left(  1,2,n\right)  \in I$ (here we need that $n\geq5$). As both $1$ and
$n$ belong to $I$, we have $I=\left\{  1,2,\ldots,n\right\}  $, i.e.
$\vartheta$ is the total relation on $\mathbb{A}_{n}$.

If $a+1\leq n-3$, then a similar argument works: $n=f\left(  a,a+1,n\right)
\in I$, and then $1=f\left(  1,n-1,n\right)  \in I$, therefore $\vartheta$ is
the total relation again. The assumption $n>6$ ensures that at least one of
$a\geq4$ and $a+1\leq n-3$ holds, hence $\mathbb{A}_{n}$ is simple, as claimed.
\end{example}

Let $\mathcal{C}$ be a majority minimal clone. To simplify the notation we
will just write $1,~2$ and $3$ for the first, second and third ternary
projections respectively, and numbers greater than $3$ will denote nontrivial
elements of $\mathcal{C}^{\left(  3\right)  }$. Our next goal is to prove that
if all majority functions in $\mathcal{C}$ are cyclically symmetric, then
there is only one majority operation in the clone, i.e. $\mathcal{C}^{\left(
3\right)  }\cong\left[  m_{1}\right]  ^{\left(  3\right)  }$. In preparation,
we introduce three binary operations on the ternary part of $\mathcal{C}$.%
\begin{align*}
f\ast g  &  =f\left(  g\left(  1,2,3\right)  ,g\left(  2,3,1\right)  ,g\left(
3,1,2\right)  \right) \\
f\bullet g  &  =f\left(  g\left(  1,2,3\right)  ,2,3\right) \\
f\circledcirc g  &  =f\left(  1,g\left(  1,2,3\right)  ,g\left(  1,3,2\right)
\right)
\end{align*}

\begin{theorem}
\label{thm idemp}The operations $\ast,~\bullet$ and $\circledcirc$ are
associative, and if $\mathcal{C}$ is a majority clone, then $\mathcal{C}%
^{\left(  3\right)  }\setminus\mathcal{I}$ is closed under them. Therefore if
$\mathcal{C}^{\left(  3\right)  }$ is finite, then it contains a nontrivial
idempotent element for each of these operations.
\end{theorem}

\begin{proof}
It is easy to check that if $f$ and $g$ are majority operations, then so are
$f\ast\nolinebreak g,f\bullet\nolinebreak g$ and $f\circledcirc g$, hence
$\mathcal{C}^{\left(  3\right)  }\setminus\mathcal{I}$ is closed under these
three operations. Associativity can be checked by a routine calculation using
the three defining axioms of abstract clones. We work out the details for
$\circledcirc$, the other two cases are similar. Let us compute $\left(
f\circledcirc g\right)  \circledcirc h$ first:%
\begin{gather*}
\left(  f\circledcirc g\right)  \circledcirc h=\left(  f\circledcirc g\right)
\left(  1,h\left(  1,2,3\right)  ,h\left(  1,3,2\right)  \right)  =\\
f\left(  1,g\left(  1,h\left(  1,2,3\right)  ,h\left(  1,3,2\right)  \right)
,g\left(  1,h\left(  1,3,2\right)  ,h\left(  1,2,3\right)  \right)  \right)  .
\end{gather*}
For $f\circledcirc\left(  g\circledcirc h\right)  $ we have%
\begin{gather*}
f\circledcirc\left(  g\circledcirc h\right)  =f\left(  1,\left(  g\circledcirc
h\right)  \left(  1,2,3\right)  ,\left(  g\circledcirc h\right)  \left(
1,3,2\right)  \right)  =\\
f\left(  1,g\left(  1,h\left(  1,2,3\right)  ,h\left(  1,3,2\right)  \right)
\left(  1,2,3\right)  ,g\left(  1,h\left(  1,2,3\right)  ,h\left(
1,3,2\right)  \right)  \left(  1,3,2\right)  \right)  =\\
f\left(  1,g\left(  1,h\left(  1,2,3\right)  ,h\left(  1,3,2\right)  \right)
,g\left(  1,h\left(  1,3,2\right)  ,h\left(  1,2,3\right)  \right)  \right)  .
\end{gather*}

The last statement of the theorem follows since every finite semigroup
contains an idempotent element. Let us note that this fact is proved for the
operation $\bullet$ in Lemma~4.4 of \cite{HM} and for $\ast$ in Theorem~2.2 of
\cite{W4maj}.
\end{proof}

Now we are ready to prove the main result of this section. This theorem is an
analogue of a theorem of J.~Dudek and J. Ga\l uszka which states that if a
binary minimal clone contains finitely many nontrivial binary operations all
of which are commutative, then there is just one nontrivial binary operation
in the clone \cite{D7}.

\begin{theorem}
\label{thm symm}Let $\mathcal{C}$ be a majority minimal clone with finitely
many ternary operations. If every nontrivial ternary operation in
$\mathcal{C}$ is cyclically symmetric, then $\mathcal{C}$ contains only one
nontrivial ternary operation, hence $\mathcal{C}^{\left(  3\right)  }%
\cong\left[  m_{1}\right]  ^{\left(  3\right)  }$.
\end{theorem}

\begin{proof}
Let $\mathcal{C}^{\left(  3\right)  }=\left\{  1,2,\ldots,n\right\}  $, where
$1,2,3$ are the ternary projections as before. First let us assume that there
is no totally symmetric majority function in $\mathcal{C}$, i.e.
$\sigma\!\left(  f\right)  =A_{3}$ for all $f\geq4$. By
Theorem~\ref{thm idemp} there is a nontrivial $\circledcirc$-idempotent, say
$4\circledcirc4=4$. Since $4$ is not invariant under the transposition
$\left(  23\right)  $, the element $4\left(  1,3,2\right)  $ is different from
$4$, thus we may suppose without loss of generality that $4\left(
1,3,2\right)  =5$. We have $4\left(  1,4,5\right)  =4\circledcirc4=4$, hence
$4\left(  1,4,5\right)  =4\left(  4,5,1\right)  =4\left(  5,1,4\right)  =4$
because $4$ is cyclically symmetric. We can compute $4\left(  1,5,4\right)  $
as well, using the associativity of composition:%
\[
4\left(  1,5,4\right)  =4\left(  1\left(  1,3,2\right)  ,4\left(
1,3,2\right)  ,5\left(  1,3,2\right)  \right)  =4\left(  1,4,5\right)  \left(
1,3,2\right)  =4\left(  1,3,2\right)  =5.
\]
Thus we have $4\left(  1,5,4\right)  =4\left(  5,4,1\right)  =4\left(
4,1,5\right)  =5$ , therefore $4$ preserves $\left\{  1,4,5\right\}  $, and
its restriction to this set is isomorphic to $m_{3}$. However, $m_{3}$
generates majority operations that are not cyclically symmetric (see
Table~\ref{table 3maj}), and this contradicts our assumption that every
nontrivial ternary operation of $\mathcal{C}$ is cyclically symmetric.

This contradiction shows that $\mathcal{C}$ must contain at least one totally
symmetric majority function. If $f$ and $g$ are totally symmetric, then
$f\bullet g$ is invariant under the transposition $\left(  23\right)  $:%
\begin{gather*}
\left(  f\bullet g\right)  \left(  1,3,2\right)  =f\left(  g\left(
1,2,3\right)  ,2,3\right)  \left(  1,3,2\right)  =\\
f\left(  g\left(  1,3,2\right)  ,3,2\right)  =f\left(  g\left(  1,2,3\right)
,2,3\right)  =f\bullet g.
\end{gather*}
Since $f\bullet g$ is nontrivial, it is also cyclically symmetric, hence
$\sigma\!\left(  f\bullet g\right)  =S_{3}$. Thus totally symmetric majority
functions form a finite semigroup under $\bullet$, so there is a totally
symmetric $f\in\mathcal{C}^{\left(  3\right)  }\setminus\mathcal{I}$ with
$f\bullet f=f$. Then $f$ satisfies the identities in (\ref{e M1}), hence
$\left[  f\right]  ^{\left(  3\right)  }\cong\left[  m_{1}\right]  ^{\left(
3\right)  }$. By the minimality of $\mathcal{C}$ we have $\left[  f\right]
=\mathcal{C}$, and this proves the theorem.
\end{proof}

\begin{corollary}
\label{orbit}If $\mathcal{C}$ is a majority minimal clone with $2\leq
\left\vert \mathcal{C}^{\left(  3\right)  }\setminus\mathcal{I}\right\vert
<\aleph_{0}$, then the action of $S_{3}$ on $\mathcal{C}^{\left(  3\right)
}\setminus\mathcal{I}$ has an orbit with at least $3$ elements.
\end{corollary}

\begin{proof}
By the previous theorem there is a nontrivial operation $f\in\mathcal{C}%
^{\left(  3\right)  }$ which is not cyclically symmetric. Thus $\sigma
\!\left(  f\right)  $ has at most $2$ elements, and therefore the size of the
orbit of $f$ is $6/\left\vert \sigma\!\left(  f\right)  \right\vert \geq3$.
\end{proof}

\section{Minimal clones with at most seven ternary
operations\label{sect atmostfour}}

In this section we are going to prove the following characterization of
majority minimal clones with at most seven ternary operations.

\begin{theorem}
\label{thm few}If $\mathcal{C}$ is a majority minimal clone with at most seven
ternary operations, then $\mathcal{C}^{\left(  3\right)  }$ is isomorphic to
either $\left[  m_{1}\right]  ^{\left(  3\right)  }$ or $\left[  m_{2}\right]
^{\left(  3\right)  }$.
\end{theorem}

Since there are three ternary projections, the clones under consideration
contain $1,2,3$ or 4 majority operations. Theorem~\ref{thm maj=1} describes
the minimal clones with one majority operation, and from Corollary \ref{orbit}
we see immediately that there is no minimal clone with exactly two majority
operations. We will deal with the cases of three and four majority operations
in two separate lemmas.

\begin{lemma}
\label{thm maj=3}If $\mathcal{C}$ is a minimal clone with three majority
operations, then $\mathcal{C}^{\left(  3\right)  }$ is isomorphic to $\left[
m_{2}\right]  ^{\left(  3\right)  }$.
\end{lemma}

\begin{proof}
Let $\mathcal{C}$ be a minimal clone with three majority functions, and let
$\mathcal{C}^{\left(  3\right)  }=\left\{  1,2,3,4,5,6\right\}  $, where
$1,2,3$ are the ternary projections. Considering the orbits of the action of
$S_{3}$ on $\left\{  4,5,6\right\}  $ we see by Corollary \ref{orbit} that the
only possibility is that there is just one orbit, i.e. any two nontrivial
ternary operations can be obtained form each other by cyclic permutations of
variables. We can suppose that $4\left(  2,3,1\right)  =5$ and $5\left(
2,3,1\right)  =6$ (and then $6\left(  2,3,1\right)  =4$).

Any composition of majority operations is again a majority operation,
therefore the set $\mathcal{C}^{\left(  3\right)  }\setminus\mathcal{I}%
=\left\{  4,5,6\right\}  $ is preserved by $4$. This implies that every
operation in $\mathcal{C}$ preserves $\left\{  4,5,6\right\}  $, since
$\mathcal{C}=\left[  4\right]  .$ Thus we have a clone homomorphism%
\[
\varphi:\mathcal{C}\rightarrow\mathcal{O}_{\left\{  4,5,6\right\}
},\,f\mapsto f|_{\left\{  4,5,6\right\}  }.
\]

We claim that $\varphi$ is injective on $\left\{  1,2,3,4,5,6\right\}  $.
Clearly it suffices to show that $\varphi\left(  4\right)  \neq\varphi\left(
5\right)  \neq\varphi\left(  6\right)  \neq\varphi\left(  4\right)  $. We
prove the first unequality, the other two are similar. Let us compute
$5\left(  4,5,6\right)  $ using the associativity of composition:%
\begin{gather*}
5\left(  4,5,6\right)  =4\left(  2,3,1\right)  \left(  4,5,6\right)  =4\left(
5,6,4\right)  =\\
4\left(  4\left(  2,3,1\right)  ,5\left(  2,3,1\right)  ,6\left(
2,3,1\right)  \right)  =4\left(  4,5,6\right)  \left(  2,3,1\right)  .
\end{gather*}
Since $4\left(  4,5,6\right)  \in\left\{  4,5,6\right\}  $ and none of these
three elements are invariant under the permutation $\left(  231\right)  $, we
have $5\left(  4,5,6\right)  =4\left(  4,5,6\right)  \left(  2,3,1\right)
\neq4\left(  4,5,6\right)  $. Thus $4|_{\left\{  4,5,6\right\}  }%
\neq5|_{\left\{  4,5,6\right\}  }$ as claimed.

Now we see that $\mathcal{C}^{\left(  3\right)  }$ is isomorphic to its image
under $\varphi$, which is the ternary part of a minimal clone on a
three-element set. Therefore $\mathcal{C}^{\left(  3\right)  }\cong\left[
m_{i}\right]  ^{\left(  3\right)  }$ for some $i\in\left\{  1,2,3\right\}  $.
The cardinality of $\mathcal{C}^{\left(  3\right)  }$ is $6$, so we must have
$i=2$, and the lemma is proved.
\end{proof}

\begin{remark}
The previous lemma can be formulated in terms of algebras and varieties as
follows. Let $\mathcal{M}_{2}$ be the variety defined by the three-variable
identities satisfied by $\left(  \left\{  1,2,3\right\}  ;m_{2}\right)  $. If
$f$ is a majority operation on a set $A$, then $\left[  f\right]  $ is a
minimal clone with exactly three majority operations iff $\left(  A;f\right)
$ is term-equivalent to an element of $\mathcal{M}_{2}\setminus\mathcal{M}%
_{1}$. Note that no two different subvarieties of $\mathcal{M}_{2}$ are
term-equivalent, since for any $\mathbb{A=}\left(  A;f\right)  \in
\mathcal{M}_{2}$ the basic operation $f$ is the only nontrivial ternary
function in $\operatorname{Clo}\mathbb{A}$ which is invariant under the
transposition $\left(  23\right)  $. This means that in order to show that
there are infinitely many nonisomorphic minimal clones with three majority
operations, it suffices to verify that the variety $\mathcal{M}_{2}$ has
infinitely many subvarieties that are not contained in $\mathcal{M}_{1}$. If
$d_{A}$ is the dual discriminator function on a set $A$ with at least three
elements, then $\left(  A;d_{A}\left(  z,y,x\right)  \right)  \in
\mathcal{M}_{2}\setminus\mathcal{M}_{1}$, and by J\'{o}nsson's lemma we have
$\left(  B;d_{B}\left(  z,y,x\right)  \right)  \notin\operatorname{HSP}\left(
A;d_{A}\left(  z,y,x\right)  \right)  $ if $A$ is finite and $\left\vert
A\right\vert <\left\vert B\right\vert $. Thus the algebras $\left(
A;d_{A}\left(  z,y,x\right)  \right)  $ with $A=\left\{  1,2,\ldots,n\right\}
$ and $n\geq3$ generate pairwise different subvarieties of $\mathcal{M}_{2}$
that are not contained in $\mathcal{M}_{1}$.
\end{remark}

\begin{lemma}
\label{thm maj=4}There is no minimal clone with four majority operations.
\end{lemma}

\begin{proof}
Let us suppose that $\mathcal{C}$ is a minimal clone with four majority
functions, and let $\mathcal{C}^{\left(  3\right)  }=\left\{
1,2,3,4,5,6,7\right\}  $, with $1,2,3$ being the ternary projections.
Corollary \ref{orbit} shows that there are two orbits under the action of
$S_{3}$ on $\left\{  4,5,6,7\right\}  $: a three-element and a one-element
orbit. Thus one of the four nontrivial operations is totally symmetric, the
other three operations have two-element invariance groups, and the latter
three functions can be obtained from each other by cyclic permutations of
their variables. We may assume without loss of generality that $7$ is totally
symmetric, and $4$, $5$ and $6$ are invariant under the transpositions
$\left(  23\right)  $, $\left(  13\right)  $ and $\left(  12\right)  $
respectively. Then we must have $4\left(  2,3,1\right)  =5$, $5\left(
2,3,1\right)  =6$ and $6\left(  2,3,1\right)  =4$.

Since any composition of majority operations is nontrivial, every operation in
$\mathcal{C}$ preserves $\left\{  4,5,6,7\right\}  $. Restricting to this set,
we obtain (the ternary part of) a minimal clone on a four-element set. The
operation $7\left(  4,5,6\right)  $ is easily seen to be totally symmetric:
applying a permutation to $7\left(  4,5,6\right)  $ will just permute $4,5$
and $6$ in the arguments of $7$, and this has no effect on the final value, as
$7$ is totally symmetric. Since the only totally symmetric operation in
$\mathcal{C}^{\left(  3\right)  }$ is $7$, we must have $7\left(
4,5,6\right)  =7$. This means that the restriction of $7$ to $\left\{
4,5,6,7\right\}  $ is a totally symmetric minimal majority operation that is
not conservative. Now Theorem~\ref{thm 3.14} implies that $7|_{\left\{
4,5,6,7\right\}  }$ is isomorphic to $M_{1}$, so $7\left(  a,b,c\right)  =7$
for any pairwise distinct $a,b,c\in\left\{  4,5,6,7\right\}  $. Moreover,
since $M_{1}$ does not generate any majority operation but itself, the
operations $4,5,6,7$ coincide with each other on $\left\{  4,5,6,7\right\}  $:%
\begin{equation}
f\left(  a,b,c\right)  =7\text{ if }f,a,b,c\in\left\{  4,5,6,7\right\}  \text{
and }a,b,c\text{ are pairwise distinct.} \label{4567}%
\end{equation}
In particular, we have $6\left(  6,4,5\right)  =7$, and taking into account
that $4$ and $5$ are obtained from $6$ by cyclic permutations of variables,
this can be written as $6\ast6=7$.

In what follows, we will compute many more compositions until we get a
contradiction by constructing a nontrivial ternary operation in $\mathcal{C}$
which is different from $4,5,6$ and $7$.

The operation $7\left(  1,2,7\right)  $ is invariant under the transposition
$\left(  12\right)  $, hence it is either $6$ or $7$. The latter is
impossible, since $7\left(  1,2,7\right)  =7$ implies that $7$ satisfies the
identities in (\ref{e M1}), and then the clone generated by $7$ would contain
just one nontrivial ternary operation. Thus we have $7\left(  1,2,7\right)
=6$, and by the total symmetry of $7$ it follows that%
\begin{equation}
7\left(  1,2,7\right)  =7\left(  7,1,2\right)  =7\left(  2,7,1\right)  =6.
\label{7(127)=6}%
\end{equation}

Let us now consider the values of $6$ on $\left(  1,2,7\right)  ,\left(
2,7,1\right)  ,\left(  7,1,2\right)  $. We have $6\left(  1,2,7\right)
\in\left\{  6,7\right\}  $ since $6\left(  1,2,7\right)  $ is invariant under
$\left(  12\right)  $. Applying this transposition to $6\left(  2,7,1\right)
$ we obtain $6\left(  7,1,2\right)  $:%
\[
6\left(  2,7,1\right)  \left(  2,1,3\right)  =6\left(  1,7,2\right)  =6\left(
7,1,2\right)  .
\]
Therefore either both $6\left(  2,7,1\right)  $ and $6\left(  7,1,2\right)  $
are equal to $6$ or $7$, or one of them is $4,$ the other one is $5$. The
resulting eight possibilities are summarized in the following table.%
\begin{equation}%
\begin{tabular}
[c]{ccccccccc}\hline
\multicolumn{1}{|c}{$\!6\left(  1,2,7\right)  $} & \multicolumn{1}{|c}{$\!6$}
& \multicolumn{1}{|c}{$\!6$} & \multicolumn{1}{|c}{$\!6$} &
\multicolumn{1}{|c}{$\!6$} & \multicolumn{1}{|c}{$\!7$} &
\multicolumn{1}{|c}{$\!7$} & \multicolumn{1}{|c}{$\!7$} &
\multicolumn{1}{|c|}{$\!7$}\\
\multicolumn{1}{|c}{$\!6\left(  2,7,1\right)  $} & \multicolumn{1}{|c}{$\!7$}
& \multicolumn{1}{|c}{$\!6$} & \multicolumn{1}{|c}{$\!4$} &
\multicolumn{1}{|c}{$\!5$} & \multicolumn{1}{|c}{$\!7$} &
\multicolumn{1}{|c}{$\!6$} & \multicolumn{1}{|c}{$\!4$} &
\multicolumn{1}{|c|}{$\!5$}\\
\multicolumn{1}{|c}{$\!6\left(  7,1,2\right)  $} & \multicolumn{1}{|c}{$\!7$}
& \multicolumn{1}{|c}{$\!6$} & \multicolumn{1}{|c}{$\!5$} &
\multicolumn{1}{|c}{$\!4$} & \multicolumn{1}{|c}{$\!7$} &
\multicolumn{1}{|c}{$\!6$} & \multicolumn{1}{|c}{$\!5$} &
\multicolumn{1}{|c|}{$\!4$}\\\hline
& $\!$ & $\!\uparrow$ & $\!$ & $\!$ & $\!$ & $\!\uparrow$ & $\!$ & $\!$%
\end{tabular}
\ \ \ \ \ \ \ \label{6<127>}%
\end{equation}

Let us consider any of the eight columns, and let $a,b,c$ be the elements in
this column. Then using the fact that $7=6\ast6$, we obtain
\[
7\left(  1,2,7\right)  =6\left(  6\left(  1,2,7\right)  ,6\left(
2,7,1\right)  ,6\left(  7,1,2\right)  \right)  =6\left(  a,b,c\right)  .
\]
For the two columns marked by the arrows this gives $7\left(  1,2,7\right)
=6$ by the majority rule. Similarly, for the first and the fifth column the
majority rule yields $7\left(  1,2,7\right)  =7$, and in the remaining four
cases we get $7\left(  1,2,7\right)  =7$ again, according to (\ref{4567}).
However, we already know from (\ref{7(127)=6}) that $7\left(  1,2,7\right)
=6$, so one of the two possibilities indicated by the arrows takes place. In
both cases we have
\begin{equation}
6\left(  2,7,1\right)  =6. \label{6(271)=6}%
\end{equation}

Now we go on to collect some information about the function $7$. For the
reader's convenience, we put the number of the equation being used over the
equality sign in the following calculations.

First of all, using (\ref{7(127)=6}) and (\ref{6(271)=6}) we obtain
\[
7\left(  6,2,7\right)  \overset{\left(  \ref{7(127)=6}\right)  }{=}7\left(
7,1,2\right)  \left(  2,7,1\right)  \overset{\left(  \ref{7(127)=6}\right)
}{=}6\left(  2,7,1\right)  \overset{\left(  \ref{6(271)=6}\right)  }%
{=}6\text{.}%
\]
Permuting variables we get
\begin{subequations}
\label{7(3457)}%
\begin{align}
7\left(  4,3,7\right)   &  =7\left(  6,2,7\right)  \left(  2,3,1\right)
=6\left(  2,3,1\right)  =4;\label{7(374)=4}\\
7\left(  5,3,7\right)   &  =7\left(  6,2,7\right)  \left(  1,3,2\right)
=6\left(  1,3,2\right)  =5. \label{7(375)=5}%
\end{align}

We already know from (\ref{4567}) that $7\left(  4,5,7\right)  =7$, and let us
suppose for a moment that $7\left(  4,5,3\right)  =7$. Then (\ref{7(3457)})
shows that $7$ preserves $\left\{  3,4,5,7\right\}  $, and its restriction to
this four-element set is a totally symmetric nonconservative minimal majority
function. Therefore it is isomorphic to $M_{1}$ by Theorem \ref{thm 3.14}.
However, this is clearly not the case. This contradiction shows that $7\left(
4,5,3\right)  \neq7$. Let us observe that $7\left(  4,5,3\right)  \left(
2,1,3\right)  =7\left(  5,4,3\right)  =7\left(  4,5,3\right)  $, i.e.
$7\left(  4,5,3\right)  $ is invariant under the transposition $\left(
12\right)  $. Since $6$ and $7$ are the only nontrivial functions in our clone
which are invariant under $\left(  12\right)  $, we must have
\end{subequations}
\begin{equation}
7\left(  4,5,3\right)  =6. \label{7(345)=6}%
\end{equation}

Next we calculate the value of $6\left(  4,5,3\right)  $:%
\begin{equation}
6\left(  4,5,3\right)  \overset{\left(  \ref{7(127)=6}\right)  }{=}7\left(
1,2,7\right)  \left(  4,5,3\right)  \overset{\left(  \ref{7(345)=6}\right)
}{=}7\left(  4,5,6\right)  \overset{\left(  \ref{4567}\right)  }{=}7.
\label{6(453)=7}%
\end{equation}
Note that $6\left(  3,4,5\right)  \left(  2,1,3\right)  =6\left(
3,5,4\right)  =6\left(  5,3,4\right)  $, hence similarly to the previous
table, we can list the possible behaviours of $6$ on $\{\left(  4,5,3\right)
,\allowbreak\left(  5,3,4\right)  ,\left(  3,4,5\right)  \}$.%
\begin{equation}%
\begin{tabular}
[c]{ccccc}\hline
\multicolumn{1}{|c}{$\!6\left(  4,5,3\right)  $} & \multicolumn{1}{|c}{$\!7$}
& \multicolumn{1}{|c}{$\!7$} & \multicolumn{1}{|c}{$\!7$} &
\multicolumn{1}{|c|}{$\!7$}\\
\multicolumn{1}{|c}{$\!6\left(  5,3,4\right)  $} & \multicolumn{1}{|c}{$\!7$}
& \multicolumn{1}{|c}{$\!6$} & \multicolumn{1}{|c}{$\!5$} &
\multicolumn{1}{|c|}{$\!4$}\\
\multicolumn{1}{|c}{$\!6\left(  3,4,5\right)  $} & \multicolumn{1}{|c}{$\!7$}
& \multicolumn{1}{|c}{$\!6$} & \multicolumn{1}{|c}{$\!4$} &
\multicolumn{1}{|c|}{$\!5$}\\\hline
& $\!$ & $\!\uparrow$ & $\!$ & $\!$%
\end{tabular}
\ \ \ \ \ \ \ \ \ \ \ \ \label{6<345>}%
\end{equation}
We can read $7\left(  4,5,3\right)  $ from this table in the same way as we
read $7\left(  1,2,7\right)  $ from (\ref{6<127>}). We see that $7\left(
4,5,3\right)  =7$ in three of the four cases. However, we already know that
$7\left(  4,5,3\right)  \overset{\left(  \ref{7(345)=6}\right)  }{=}6$, so the
only possibility is the one marked by the arrow.

Finally, to reach the desired contradiction, let us consider $6\left(
2,3,6\right)  $. Denoting this composition by $f$, we show that $f\left(
4,5,3\right)  =5$:%
\[
f\left(  4,5,3\right)  =6\left(  2,3,6\right)  \left(  4,5,3\right)
\overset{\left(  \ref{6(453)=7}\right)  }{=}6\left(  5,3,7\right)
\overset{\left(  \ref{7(127)=6}\right)  }{=}7\left(  1,2,7\right)  \left(
5,3,7\right)  \overset{\left(  \mathrm{\ref{7(375)=5}}\right)  }{=}7\left(
5,3,5\right)  =5\text{.}%
\]

The operation $f$ is nontrivial, but it does not coincide with any of
$4,~5,~6$ or $7$, because the value of these functions on $\left(
4,5,3\right)  $ is different from $5$. Indeed, we have%
\begin{align*}
4  &  \left(  4,5,3\right)  \overset{%
\hphantom{\left( 4.6 \right)}%
}{=}6\left(  5,3,4\right)  \overset{\left(  \ref{6<345>}\right)  }{=}6;\\
5  &  \left(  4,5,3\right)  \overset{%
\hphantom{\left( 4.6 \right)}%
}{=}6\left(  3,4,5\right)  \overset{\left(  \ref{6<345>}\right)  }{=}6;\\
6  &  \left(  4,5,3\right)  \overset{\left(  \ref{6(453)=7}\right)  }{=}7;\\
7  &  \left(  4,5,3\right)  \overset{\left(  \ref{7(345)=6}\right)  }{=}6.
\end{align*}
Thus we have more than four majority operations in our clone, and this
contradiction completes the proof.
\end{proof}


\begin{thebibliography}{999}                                                                                              %


\bibitem[BL]{BL}G. Birkhoff, J. D. Lipson, \textit{Heterogeneous algebras,} J.
Combinatorial Theory \textbf{8} (1970), 115--133.

\bibitem[Cs1]{Cs3all}B. Cs\'{a}k\'{a}ny, \textit{All minimal clones on the
three-element set,} Acta Cybernet. \textbf{6} (1983), no. 3, 227--238.

\bibitem[Cs2]{Cscons}B. Cs\'{a}k\'{a}ny, \textit{On conservative minimal
operations, }Lectures in Universal Algebra (Szeged, 1983), Colloq. Math. Soc.
J\'{a}nos Bolyai, \textbf{43}, North-Holland, Amsterdam, 1986, 49--60.

\bibitem[Cs3]{Csminicourse}B. Cs\'{a}k\'{a}ny, \textit{Minimal clones---a
minicourse,} Algebra Universalis \textbf{54} (2005), no. 1, 73--89.

\bibitem[CsG]{CsG}B. Cs\'{a}k\'{a}ny, T. Gavalcov\'{a}, \textit{Finite
homogeneous algebras I, }Acta Sci. Math. (Szeged) \textbf{42} (1980), no. 1-2, 57--65.

\bibitem[DG]{D7}J. Dudek, J. Ga\l uszka, \textit{Theorems of idempotent
commutative groupoids,} Algebra Colloq. \textbf{12} (2005), no. 1, 11--30

\bibitem[FP]{FP}E. Fried, A. F. Pixley, \textit{The dual discriminator
function in universal algebra,} Acta Sci. Math. (Szeged) \textbf{41} (1979),
no. 1-2, 83--100.

\bibitem[HM]{HM}D. Hobby, R. McKenzie, \textit{The structure of finite
algebras, }Contemporary Mathematics \textbf{76}, American Mathematical
Society, Providence, RI, 1988.

\bibitem[Ke]{KK}K. A. Kearnes, \textit{Minimal clones with abelian
representations,} Acta Sci. Math. (Szeged) \textbf{61} (1995), no. 1-4, 59--76.

\bibitem[LP]{LLPPP}L. L\'{e}vai, P. P. P\'{a}lfy, \textit{On binary minimal
clones,} Acta Cybernet. \textbf{12} (1996), no. 3, 279--294.

\bibitem[PK]{PK}R. P\"{o}schel, L. A. Kalu\v{z}nin, \textit{Funktionen- und
Relationenalgebren, }Mathematische Monographien, VEB Deutscher Verlag der
Wissenschaften, Berlin, 1979. (German)

\bibitem[Qu1]{Qcat}R. W. Quackenbush, \textit{Some remarks on categorical
algebras}, Algebra Universalis \textbf{2} (1972), 246.

\bibitem[Qu2]{Qsurv}R. W. Quackenbush, \textit{A survey of minimal clones},
Aequationes Mathematicae \textbf{50} (1995), 3-16.

\bibitem[Ro]{R5typ}I. G. Rosenberg, \textit{Minimal clones I. The five types,}
Lectures in Universal Algebra (Szeged, 1983), Colloq. Math. Soc. J\'{a}nos
Bolyai, \textbf{43}, North-Holland, Amsterdam, 1986, 405--427.

\bibitem[Sw]{Sw}S. \'{S}wierczkowski, \textit{Algebras which are independently
generated by every }$\mathit{n}$\textit{\ elements,} Fund. Math. \textbf{49}
(1960/1961), 93--104.

\bibitem[Sz]{SzAclUA}\'{A}. Szendrei, \textit{Clones in Universal Algebra,}
S\'{e}minaire de Math\'{e}matiques Sup\'{e}rieures, \textbf{99}, Presses de
L'Universit\'{e} de Montr\'{e}al, 1986.

\bibitem[Ta]{Taylor}W. Taylor, \textit{Characterizing Mal'cev conditions,}
Algebra Universalis \textbf{3} (1973), 351--397.

\bibitem[Wa]{W4maj}T. Waldhauser, \textit{Minimal clones generated by majority
operations,} Algebra Universalis \textbf{44} (2000), no. 1-2, 15--26.
\end{thebibliography}
\end{document}